    \documentclass[default]{sn-jnl}

\usepackage{graphicx}%
\usepackage{multirow}%
\usepackage{amsmath,amssymb,amsfonts}%
\usepackage{amsthm}%
\usepackage{mathrsfs}%
\usepackage[title]{appendix}%
\usepackage{xcolor}%
\usepackage{textcomp}%
\usepackage{manyfoot}%
\usepackage{booktabs}%
\usepackage{algorithm}%
\usepackage{algorithmicx}%
\usepackage{algpseudocode}%
\usepackage{listings}%

 \usepackage[square,numbers,sort]{natbib}
  \usepackage{amsfonts}
  \usepackage[T1]{fontenc}    
  \usepackage{anysize}
\usepackage{setspace}
\usepackage{tfrupee}
\usepackage{adjustbox}
\usepackage{tabularx}
\usepackage{subcaption}

\newtheorem{example}{Example}

\newtheorem{definition}{Definition}
\newtheorem{Lemma}{Lemma}

  \newcommand{\nn}{\nonumber}
  
\newtheorem{theorem}{Theorem}
\usepackage{lmodern}
\usepackage{anyfontsize}
\begin{document}

\title[STGPC]{Generating fractal functions associated with Suzuki iterated function systems}


\author*[1]{\fnm{Mridul} \sur{Patel}}\email{mridul.patel@student.rmit.edu.au }

\author[1]{\fnm{G.} \sur{Verma}}\email{geetika.verma@rmit.edu.au}
\equalcont{These authors contributed equally to this work.}

\author[1]{\fnm{A.} \sur{Eberhard}}\email{andy.eberhard@rmit.edu.au}
\equalcont{These authors contributed equally to this work.}

\author[1]{\fnm{A.} \sur{Rao}}\email{asha.rao@rmit.edu.au}
\equalcont{These authors contributed equally to this work.}

\affil[1]{\orgdiv{School of Mathematical Sciences}, \orgname{RMIT University}, \orgaddress{\street{124, LaTrobe Street}, \city{Melbourne}, \postcode{3000}, \state{VIC}, \country{Australia}}}





\abstract{This article constructs a fractal interpolation function, also referred to as $\alpha$-fractal function, using Suzuki-type generalized $\varphi$-contraction mappings (STGPC). The STGPC is a generalization of $\varphi$-contraction mappings. The process of constructing $\alpha$-fractal functions using the STGPC is detailed, and examples of STGPC are given. The FIF has broad applications in data analysis, finance and price prediction. We have included a case study analyzing the price volatility of spinach in the Azadpur vegetable market in New Delhi. The fractal analysis gives a unique perspective on understanding price fluctuations over a period. Finally, the box-dimensional analysis is presented to comprehend the complexity of price fluctuations.   }




\maketitle

\section{Introduction}\label{Intro}

Fractals are complex geometric shapes that can be divided into parts, with each part being a smaller-scale replica of the whole. This characteristic is known as self-similarity.  Fractals can be observed in nature, appearing in the shapes of mountains and coastlines and in the branching patterns of trees and blood vessels \citep{mandelbrot1982fractal}. The concept of fractals was popularized by Benoît B. Mandelbrot in his influential work "The Fractal Geometry of Nature" \citep{mandelbrot1982fractal}. Mandelbrot \citep{mandelbrot1982fractal} demonstrated that many complex, irregular structures could be described and analyzed mathematically using fractal geometry. This new perspective allowed more accurate modelling of natural phenomena compared to traditional Euclidean geometry \citep{feder1988fractals}.

Fractals have diverse applications across a range of fields, including physiology \cite{goldberger2002fractal}, finance \cite{mandelbrot1997fractals}, computer graphics \cite{ebert2002texturing}, geology, and geophysics \cite{turcotte1997fractals}. The study of fractals provides a mathematical framework to help us understand the complexity of the world around us, bridging abstract theory and observable reality. Barnsley explored the fractal theory, which includes the approximation theory in \citep{barnsley1986fractal}. Barnsley's research significantly influenced the study of a group of fractal interpolation functions (FIFs) called the $\alpha$-fractal functions, which are associated with the continuous function defined on a compact interval of real numbers. 

Since Barnsley's groundbreaking work, there has been growing interest in utilizing FIFs to approximate functions. A collection of non-affine fractal functions $f^\alpha$, which approximates and interpolates a continuous function $f$ on a compact interval of $\mathbb{R}$, was investigated by \citep{chand2006generalized, jha2021approximation,navascues2010fractal,navascues2005fractal}. 
FIFs have found applications in a variety of areas. For example, Verma and Kumar \cite{verma2024fractal} analysed the performance of financial institutions after mergers and acquisitions. They \cite{verma2024fractal} obtained coefficients using a regression model, subsequently using these coefficients in an FIF to study performance. Further applications of FIF can be found in \citep{wang2019novel,raubitzek2021fractal, manousopoulos2023financial,agrawal2023dimensional,husain2024novel}.

 Fractals exhibit distinct geometric properties such as self-similarity and scale-free characteristics. These properties are often analyzed through dimensional analysis of fractal patterns. The fractal dimension is crucial for measuring small parts of fractals \cite{ida1998image}, identifying shapes \cite{neil1997shape}, and quantifying the roughness of fractal surfaces \cite{nayak2019analysing}. Varying definitions of fractal dimensions are applicable in specific contexts. Hence, various methods have been suggested to estimate the fractal dimension, including the Higuchi fractal dimension \cite{higuchi1988approach}, texture fractal dimension \cite{chaudhuri1995texture}, aggregates fractal dimension \cite{jiang1991fractal}, and the more widely used box-counting dimension \cite{li2009improved}. The box-counting dimension is frequently employed in various applications due to its simplicity and automatic computability, as seen in (\citep{foroutan1999advances,peitgen2004chaos,xu2006new,jiang2012box,so2017enhancement}).

This study concentrates on constructing the $\alpha$-fractal function using the Suzuki-type $\varphi$-generalized contraction mapping (STGPC) on a complete metric space. This new class of mapping is a generalization of the Banach contraction (BC) mapping. Pant's paper \cite{pant2018fixed} established some fixed-point results on a complete metric space using the STGPC. Here, we describe the construction of a FIF from an iterated function system (IFS), as introduced in  \cite{barnsley1986fractal}  and detailed in \citep{barnsley1985iterated}. Fractals are formed through Iterated Function Systems (IFS), which involve a finite number of Banach contractions. Fractal functions display similar behaviour, indicating the noteworthy influence of the BC principle in the development of diverse FIFs. Existing literature explores numerous other contraction mappings and fixed-point theorems that establish the existence and uniqueness of fixed points  \citep{sahu2010k,uthayakumar2014iterated,thangaraj2024generation}. This study introduces the fractal interpolation function on a complete metric space using the STGPC. Moreover, we construct the fractal operator through the STGPC using the IFS. This fractal operator can approximate any given continuous function using an $\alpha$-fractal function.

The rest of the paper is organised as follows: In Section \ref{bg}, we will describe the fundamental concepts of IFS and $\alpha$-fractal functions. Section \ref{FIF_S}  will give the definition of Suzuki-type generalized $\varphi$-contraction mapping (STGPC) with some concrete examples with some fixed point results to construct the IFS involving STGPC presented. The construction of IFS using the STGPC map is presented in subsection \ref{IFS_construction}. Moreover, the construction of $\alpha$-FIF using STGPC and the box-dimensional analysis are presented in subsections \ref{S_FIF} and \ref{B_D}, respectively. Finally, in Section \ref{cs}, we present a case study that examines the fluctuations of spinach in a highly volatile market using the fractal interpolation function. The study focuses on the price data from the Azadpur market in New Delhi, India, from September 2023 to July 2024. We also include a box-dimension analysis to understand the complexity of the price fluctuations during this time frame. By using a fractal representation of the average monthly data, we gain a unique perspective on understanding the high volatility of the prices.


\section{Background and Preliminaries}\label{bg}

This section provides a survey of the relevant definitions and results related to contraction mappings and iterated function systems (IFS). 
 
\subsection{Iterated function systems (IFS)}
Consider $A=[y_0,y_P]$ and $A_p=[y_{p-1},y_p]$. Let the given data set be $\lbrace (y_p,z_p)  \in A\times \mathbb{R}:p=0,1,\ldots,P\rbrace$ with strictly increasing abscissae.
Let $L_p : A \rightarrow A_p$ be the $p$ contractive homeomorphisms defined by $L_p(y) = a_p y+b_p$, where $a_p$ and $b_p$ are unknown parameters, determined in (\cite{barnsley1986fractal}) using the following conditions, 
\begin{align}\label{affine}
    L_p(y_0)=y_{p-1},~~~L_p(y_P)=y_p, ~~\forall~ p=1,2,\ldots,P.
\end{align}
Let $F_p: K (:= A \times \mathbb{R}) \rightarrow \mathbb{R}$ denote the functions that are continuous in the first variable and contractive in the second variable satisfying
\begin{align}\label{F_p}
F_p(y_0, z_0)= z_{p-1},~~~F_p(y_P, z_P )=z_p.  
\end{align}
Define the map $w_p : K (:= A \times \mathbb{R})\rightarrow K 
$ by $w_p(y, z) = (L_p(y), F_p(y, z))$ for all $p \in P$. Then, the system defined by 
\begin{align}\label{IFS}
\lbrace K;w_p: p=1,2,\ldots, P \rbrace
\end{align}
is called the IFS with a finite number of contraction mappings, and the attractor of the IFS is a unique invariant set $\mathcal{A}$ such that $\mathcal{A}=  \displaystyle \bigcup^{P}_{p=1} w_p (\mathcal{A})$. The set $\mathcal{A}$ is the graph of a continuous function $g: A \rightarrow \mathbb{R} $ which satisfies $g(y_p)=y_p, p=1,2,\ldots, P$. The function $g$ is referred to as the FIF or simply the fractal function associated with the IFS \eqref{IFS}, which satisfies the equation, 
\begin{align}\label{f_Fn}
g(L_p(y))=F_p(y,g(y)),~~~~\forall ~y\in \mathcal{A}_p,~p\in P.  
\end{align}

The most common IFS is constructed using the following maps $L_p(y)$ and $F_p(y,z)$, which are defined as follows:
\begin{align}\label{ln_Fn}
L_p(y) = a_p y+b_p~~~~F_p(y,z)= \alpha_p z + q_p(y),~~p\in P,
\end{align}
where $q_p$ are continuous functions that meet the endpoint conditions with $F_p$, and $\alpha_p$'s are the vertical scaling factors of the mappings $w_p$'s. The scale vector $\alpha=(\alpha_1,\alpha_2,\ldots,\alpha_p)$ is known as the scaling factors associated to IFS \eqref{IFS}.

The development of $\alpha$-FIF with constant and function vertical scaling factors, as detailed in \cite{wang2013fractal, barnsley1989hidden, vijender2019approximation, navascues2005fractal, akhtar2017box}, significantly enhances the versatility and flexibility of the FIF. This approach offers greater freedom in approximation. Additionally, the application of the BC principle ensures both the existence and uniqueness of the FIF, as outlined in the following theorem.

\begin{theorem}[Theorem 1, \citep{barnsley1986fractal}]
Let $\mathcal{C}$ be the space of continuous functions and $\lbrace K;w_p: p=1,2,\ldots, P \rbrace$ be an IFS as defined in \eqref{IFS}. Define  $g:A\rightarrow \mathbb{R}$ with $h(y_0)=z_0, h(y_P)=z_P$ associated with the metric  $d(h,f)=\max\lbrace|h(y)-f(y)|:y\in A \rbrace$. Then $(\mathcal{C},d)$ is a complete metric space.

Define the Read-Bajaraktarevic operator $\mathbb{T}$ on $(\mathcal{C},d)$ as $(\mathbb{T}h)(y)=F_p(L_p^{-1}(y), h(L_p^{-1}(y))),~\forall~y \in A_p,~p=1,2,\ldots,P$. Then: 
\begin{enumerate}
    \item The iterated function system \eqref{IFS} possesses a unique attractor $G_g$, which represents the graph of $g:A\rightarrow \mathbb{R}$, interpolating the data $\lbrace (y_p,z_p)  \in A\times \mathbb{R}:p=0,1,\ldots,P\rbrace$.
    \item $\mathbb{T}$ satisfies the conditions of BC on complete metric space $(\mathcal{C},d)$ and it has a fixed point $g$ which satisfies the equation,
\end{enumerate}
\begin{align}\label{FIF}
g(y)=\alpha_p g(L_p^{-1}(y))+q_p(L_p^{-1}(y)),~~y\in A_p,~p\in P. 
\end{align}
The function $g$ in \eqref{FIF} is called the FIF associated with $\lbrace (y_p,z_p)  \in A\times \mathbb{R}:p=0,1,\ldots, P\rbrace$.
\end{theorem}
In the following section, we will construct the iterated function system and FIF with STGPC.  
\section{Fractal interpolation function using Suzuki type generalized $\varphi$- contractions}\label{FIF_S}

Numerous extensions of the BC principle can be found in the existing literature; see \citep{kannan1969some, ciric1971generalized, chatterjea1972fixed}. One of the most fascinating extensions was presented by Suzuki in \cite{suzuki2008generalized}. The fixed-point result is as follows: 
\begin{theorem}[Theorem 2, \citep{suzuki2008generalized}]
Let $(X, d)$ represent a complete metric space, and let $T$ denote a mapping on $X$. We define a non-increasing function $\theta: [0,1) \rightarrow (1/2, 1] $,
 \begin{align*}
     \theta(r)=
\begin{cases}
    1, & 0\leq r \leq \frac{\sqrt{5}-1}{2},\\
    \frac{1-r}{r^2}, & \frac{\sqrt{5}-1}{2} \leq r \leq \frac{1}{\sqrt{2}},\\
    \frac{1}{1+r}, & \frac{1}{\sqrt{2}} \leq r \leq 1.
\end{cases}
 \end{align*}
If there exists $r\in [0,1)$, such that
 \begin{align*}
     \theta(r)d(y,Ty)\leq d(y,z) \implies d(Ty,Tz)\leq r d(y,z),
 \end{align*}
 for all $y,z\in X$, then there is a unique fixed point $y\prime$ of $T$. Moreover, $\lim_n T^n y=y\prime$ for all $y\in X$.
\end{theorem}

\begin{definition}[$\varphi$-contraction mapping,
\citep{jachymski1997equivalence}] Let $(X, d)$ be a metric space and $\varphi:[0,\infty )\rightarrow [0,\infty )$ is upper semi continuous from the right on $[0,\infty )$ and satisfies $\varphi(t)<t$ for all $t>0$. A self-mapping $ T:X \rightarrow X$ is said to be $\varphi$-contractive if 
 \begin{align}\label{phi_def}
    d(Ty,Tz)\leq \varphi (d(y,z))    
 \end{align}
 for all $y, z \in X$.
\end{definition}

The generalization of the BC principle  \cite{boyd1969nonlinear} is as follows:

\begin{theorem}[Theorem 1, \citep{boyd1969nonlinear}]\label{phi_thm}
In a complete metric space $(X, d)$, if $T: X \rightarrow X$ is a $\varphi$-contractive map, then $T$ possesses a unique fixed point.   
\end{theorem}

\begin{definition}[Suzuki-type generalized $\varphi$-contraction mapping,
\citep {pant2018fixed}] Let $(X, d)$ be a metric space. A self-mapping $T:X\rightarrow X$ will be called a STGPC for all $y,z\in X$ if
\begin{align}\label{Suzuki_def}
    \frac{1}{2} d(y,Ty)\leq d(y,z)~~ \mbox{implies that}~~ d(Ty,Tz)\leq \varphi(m(y,z)), 
\end{align}
where $m(y,z)=\max \lbrace d(y,z), d(y,Ty), d(z,Tz) \rbrace$ and $\varphi:[0,\infty )\rightarrow [0,\infty )$ satisfies $\varphi(t)<t$ for all $t>0$ and $\limsup\limits_{s\rightarrow t^+} \varphi(s)<t$ for all $t>0$.
\end{definition}
The following theorem generalizes Theorem \eqref{phi_thm}. As we observe later in example \ref{ex_conti} is not $\varphi$-contractive but STGPC.  
\begin{theorem}[Theorem 2.2, \citep {pant2018fixed}]\label{fp_SC}
In a complete metric space $(X, d)$, if $T:X \rightarrow X$ is an STGPC mapping, then T possesses a unique fixed point in X.
\end{theorem}

The following examples demonstrate that the STGPC map does not have to be a Banach contraction. Since interpolation theory is closely associated with continuous functions, therefore, to apply the concept of fractal functions using the STGPC, the STGPC map must be a continuous function. Example \ref{ex_conti} confirms the existence of a continuous Suzuki generalized $\varphi$-map, which facilitates the construction of FIF through STGPC in the present study. On the other hand example \ref{ex_disconti} illustrates that an STGPC may not be a continuous map.

\begin{example}\label{ex_conti}
Let $X = \mathbb{R}$ be endowed with the metric $d(y,z)=|y-z|$. Then $(X, d)$ is a complete metric space. Define $T : X \rightarrow X$ and $ \varphi : [0, \infty) \rightarrow [0, \infty)$ defined by $\varphi(t)=\frac{t}{2}$.
\begin{align*}
T(y)=
\begin{cases}
    ~~0 &if~ y\leq 4,\\
    2y-8 & if ~y\in [4,5],\\
    -\frac{y}{2}+\frac{9}{2} & if ~y\in [5,7],\\
    -y+8& if ~y\in [7,8],\\
    ~~~0 &if~~ y\geq 8.
\end{cases}
\end{align*}
We will verify that $T$ satisfies the Suzuki-type generalized $\varphi$-contraction condition \eqref{suzuki_condtion} and $T$ does not satisfy theorem \ref{phi_thm} and hence not a $\varphi$-contraction map. 

It is clear that $\varphi(t) = \frac{t}{2} < t $ for $t > 0$.
\begin{align}\label{suzuki_condtion}
\mbox{If}~~\frac{1}{2} d(y, Ty) \leq d(y, z)~~ then~~ d(Ty, Tz) \leq \varphi(\max\{d(y, z), d(y, Ty), d(z, Tz)\}    
\end{align}
is satisfied for all $y, z \in \mathbb{R}$. Consider $\mathcal{M}=\max \lbrace d(y, z), d(y, Ty), d(z, Tz) \rbrace$ from equation \eqref{suzuki_condtion}.
We will show that $T$ is a Suzuki-type generalized $\varphi$-contraction mapping for all the possible cases in the domain of T,
\begin{enumerate}
    \item If $y, z \leq 4$, then \eqref{suzuki_condtion} trivially holds.

    \item If $y\leq 4$ and $z \in [4,5] $, then
    \begin{align*}
        &Ty=0,~~~~~~~ Tz=2z-8,\\
        &|Ty-Tz|=|2z-8|\leq 2,~~|y-Ty|=|y|\leq 4,\\
        &|z-Tz|=|-z+8|\leq 4,~~|y-z|\leq 1,\\
        &\mathcal{M}=4 ~\mbox{then}~\varphi(\mathcal{M})=\frac{4}{2}
    \end{align*}
and \eqref{suzuki_condtion} holds.  

    \item If $y\leq 4$ and $z \in [5,7] $, then
    \begin{align*}
        &Ty=0,~~~~~~~~ Tz=\frac{-z}{2}+\frac{9}{2},\\
        &|Ty-Tz|=|2z-8|\leq 2,~~|y-Ty|=|y|\leq 4,\\
        &|z-Tz|=|-z+8|\leq 3,~~|y-z|\leq 3,\\
        &\mathcal{M}=4 ~\mbox{then}~\varphi(\mathcal{M})=\frac{4}{2}=2
    \end{align*}
and \eqref{suzuki_condtion} holds.

    \item If $y\leq 4$ and $z \in [7,8] $, then
    \begin{align*}
        &Ty=0,~~~~~~~~ Tz=-z+8,\\
        &|Ty-Tz|=|z-8|\leq 1,~~|y-Ty|=|y|\leq 4,\\
        &|z-Tz|=|2z-8|\leq 6,~~|y-z|\leq 4,\\
        &\mathcal{M}=6 ~\mbox{then}~\varphi(\mathcal{M})=\frac{6}{2}=3
    \end{align*}
and \eqref{suzuki_condtion} holds.

\item If $y \leq 4, z \geq 8 $, then \eqref{suzuki_condtion} trivially holds.
\end{enumerate}
Similarly, we can show that T satisfies the condition \eqref{suzuki_condtion} for the remaining cases. A summary of all possible cases is included in table \eqref{cases_25}. The map $T(x)$ meets the condition \eqref{suzuki_condtion} for all cases outlined in table \eqref{cases_25}. Hence, $T$ is a continuous Suzuki-type generalized $\varphi$-contraction map for all points in $X$.

\begin{table}[htp]
    \centering
    \begin{adjustbox}{width=1.1\textwidth} 
    \begin{tabularx}{\textwidth}{c*{4}{X}}\toprule
 Case No.  &  Cases  & $Ty$  & $Tz$      \\ \hline
    1.& If $y \leq 4$ and $ z \leq 4$    &$Ty=0$  &$Tz=0$   \\
    2.& If $y\leq 4$ and $z \in [4,5]$    &$Ty=0$   &$Tz=2z-8$    \\
    3.& If $y\leq 4$ and $z \in [5,7]$    &$Ty=0$   &$Tz=\frac{-z}{2}+\frac{9}{2}$    \\
     4.&If $y\leq 4$ and $z \in [7,8] $   &$Ty=0$   &$Tz=-z+8$    \\
     5.&If $y \leq 4$ and $ z \geq 8 $    &$Ty=0$   &$Tz=0$ \\
     6.&If $y\in [4,5]$ and $z\leq 4 $   &$Ty=2y-8$  & $Tz=0$   \\
     7.&If $y\in [4,5]$ and $z\in [4,5] $   &$Ty=2y-8$  &$Tz=2z-8$    \\
    8.&If $y\in [4,5]$ and $z\in [5,7] $    &$Ty=2y-8$  &$Tz=\frac{-z}{2}+\frac{9}{2}$     \\
   9.& If $y\in [4,5]$ and $z\in [7,8] $    &$Ty=2y-8$  & $Tz=-z+8$     \\
    10.& If $y\in [4,5]$ and $z\geq 8 $    &$Ty=2y-8$  & $Tz=0$    \\
    11.& If $y\in [5,7]$ and $z\leq 4 $   &$Ty=\frac{-y}{2}+\frac{9}{2}$   &$Tz=0$    \\
    12.& If $y\in [5,7]$ and $z\in [4,5] $    &$Ty=\frac{-y}{2}+\frac{9}{2}$   &$Tz=2z-8$     \\
    13.& If $y\in [5,7]$ and $z\in [5,7] $    &$Ty=\frac{-y}{2}+\frac{9}{2}$   &$Tz=\frac{-z}{2}+\frac{9}{2}$    \\
    14.& If $y\in [5,7]$ and $z\in [7,8] $    &$Ty=\frac{-y}{2}+\frac{9}{2}$   &$Tz=-z+8$    \\
    15.& If $y\in [5,7]$ and $z\geq 8 $    &$Ty=\frac{-y}{2}+\frac{9}{2}$   &$Tz=0$    \\
    16.& If $y\in [7,8]$ and $z\leq 4 $   &$Ty=-y+8$  &$Tz=0$    \\
    17.& If $y\in [7,8]$ and $z\in [4,5] $    &$Ty=-y+8$  &$Tz=2z-8$   \\
    18.& If $y\in [7,8]$ and $z\in [5,7] $    &$Ty=-y+8$  &$Tz=\frac{-z}{2}+\frac{9}{2}$    \\
    19.& If $y\in [7,8]$ and $z\in [7,8] $    &$Ty=-y+8$  &$Tz=-z+8$    \\
    20.& If $y\in [7,8]$ and $z\geq 8 $    &$Ty=-y+8$  &$Tz=0$    \\
    21.& If $ y\geq 8 $ and $z\leq 4 $    &$Ty=0$  &$Tz=0$   \\
    22.& If $ y\geq 8 $ and $z\in [4,5] $    &$Ty=0$  &$Tz=2z-8$    \\
    23.& If $ y\geq 8 $ and $z\in [5,7] $    &$Ty=0$  &$Tz=\frac{-z}{2}+\frac{9}{2}$    \\
    24.& If $ y\geq 8 $ and $z\in [7,8] $    &$Ty=0$  &$Tz=-z+8$    \\
    25.& If $ y\geq 8 $ and $z\geq 8 $    &$Ty=0$  &$Tz=0$ \\
        \bottomrule
    \end{tabularx}
    \end{adjustbox}
    \caption{ $T$ satisfies the Suzuki condition \eqref{suzuki_condtion} for all the possible cases given in the table.}
    \label{cases_25}
\end{table}

Note that, for $y=4, z=\frac{9}{2}$,
\begin{align*}
    T(y=4)=0,T\left(y=\frac{9}{2}\right)=1~\mbox{and}~|d(Ty,Tz)|=1>\frac{1}{4}=\frac{\left|\frac{9}{2}-4\right|}{2}=\varphi(d(y,z)).
\end{align*}
Therefore, $T$ does not satisfy theorem \eqref{phi_thm}. That is, $T$ is not a $\varphi$-contraction mapping. However, $T$ is a continuous Suzuki-type generalized $\varphi$-contraction map.
\end{example}
 The Suzuki map is more generalized than the Banach map class, as it allows for the existence of fixed points in maps that are not continuous. We can verify this from the following example.
 \begin{example}\label{ex_disconti}
Let $X = \lbrace 0, 2\rbrace \cup \lbrace 1, 3, 5,\ldots \rbrace$ equipped with the metric $d(x,y)=|x-y|$. Then, $(X, d)$ forms a complete metric space. Define $T : X \rightarrow X$ and $ \varphi : [0, \infty) \rightarrow [0, \infty)$ as follows:

\begin{align*}
Ty=
\begin{cases}
    4, &if~ y=5\\
    0, &if~ y=7\\
    1, &otherwise
\end{cases}
~~~~~~~~~~~\mbox{and}~~~~~~~~~~~
\varphi(t)=
\begin{cases}
    \frac{t^2}{2}, &if~ t\leq 1\\
    t-\frac{1}{3}, &if~ t>1.
\end{cases}
\end{align*}
For $y=5,z=7$, we have
\begin{align*}
d(Ty,Tz)=d(T5, T7)=d(4,0)=4>\frac{5}{3}=\varphi(2)=\varphi(d(5,7))=\varphi(d(y,z)).
\end{align*}
Therefore, $T$ does not satisfy theorem \eqref{phi_thm}. That is, $T$ is not a $\varphi$-contraction mapping.

On the other hand,
\begin{align*}
\frac{1}{2} d(y,Ty)= \frac{1}{2} d(5,T5)=\frac{1}{2} d(5,4)=\frac{1}{2}<2=d(5,7)=d(y,z).
\end{align*}
Now,
\begin{align*}
d(Ty,Tz)&=d(T5,T7)=d(4,0)=4, \\[0.4cm]
m(y,z)&=\max \lbrace d(y,z), d(y,Ty), d(z,Tz) \rbrace,\\
&=\max \lbrace d(5,7), d(5,T5), d(7,T7) \rbrace,\\
&= \max \lbrace d(5,7), d(5,4), d(7,0) \rbrace,\\
&=\max \lbrace 2,1,7 \rbrace,\\
&=7,\\[0.2cm]
\varphi(m(y,z))=\varphi(7)&=7-\frac{1}{3}=\frac{20}{3}.
\end{align*}
This implies $d(Ty,Tz)\leq \varphi(m(y,z))$. Hence, the mapping $T$ satisfies the condition \eqref{suzuki_condtion}.
\end{example}
The Suzuki-type generalized $\varphi$-contraction mapping will be utilized for constructing iterated function systems. Despite not being a $\phi$-contraction mapping, the STGPC mapping still possesses a fixed point, as discussed in previous examples. The upcoming section will cover relevant results and the construction of IFS using STGPC.
\subsection{Iterated function system using Suzuki type generalized $\varphi$-contractive mapping }\label{IFS_construction}
The following section will focus on constructing an iterated function system using a STGPC. The approach will be based on the application of specific results.

Consider $T$ a Suzuki-type generalized $\varphi$-contraction map on a metric space $(X, d)$. Subsequently, we use the following results, whereby $T$ retains its characterization as a Suzuki-type generalized $\varphi$-contraction map on $\mathcal{C}(X)$, denoting the set of all nonempty compact subsets of $X$, with respect to the Hausdorff distance. The Hausdorff metric induced by $d$ is defined by
\begin{align*}
\mathcal{D}_H(\mathcal{X},\mathcal{Y})=\max \left\{ \sup_{x\in \mathcal{X}}  \inf_{y\in \mathcal{Y}}  d(x,y) , \sup_{y\in \mathcal{Y}}  \inf_{x\in \mathcal{X}}  d(x,y)  \right\}.
\end{align*}
for all $\mathcal{X}, \mathcal{Y} \in \mathcal{C}(X)$, where $d(a,\mathcal{Y})= \inf\limits_{b\in \mathcal{Y}}  d(a,b)$. For a mapping $T:X \to X$ and $D \subseteq X$ denote $T(D) = \cup_{x \in D} T(x)$.

\begin{Lemma}[Lemma 3.1, \cite{pant2018fixed}]
 Let $(X,d)$ be a metric space and $T: X\rightarrow X$ a continuous STGPC. Define $F_T:\mathcal{C}(X)\rightarrow \mathcal{C}(X)$ such that $\forall~ D \in \mathcal{C}(X),~~F_T(D):=T(D)$. Then, for all $\mathcal{X}, \mathcal{Y} \in \mathcal{C}(X)$,
\begin{align*}
    \frac{1}{2} \mathcal{D}_H(\mathcal{X},T(\mathcal{X}))\leq \mathcal{D}_H(\mathcal{X}, \mathcal{Y})~~ \mbox{implies}~~ \mathcal{D}_H(F_T(\mathcal{X}),F_T(\mathcal{Y}))\leq \varphi(M_T(\mathcal{X}, \mathcal{Y})),   
\end{align*}
for all $\mathcal{X}, \mathcal{Y} \in \mathcal{C}(X)$, where $M_T(\mathcal{X},\mathcal{Y})=\max \lbrace \mathcal{D}_H(\mathcal{X},\mathcal{Y}), \mathcal{D}_H(\mathcal{X},T(\mathcal{X})), \mathcal{D}_H(\mathcal{Y},T(\mathcal{Y})) \rbrace$. That is, the map $F_T$ satisfies the STGPC conditions (with the same $\varphi$).
\end{Lemma}
\begin{Lemma}[Lemma 3.2\cite{pant2018fixed}]\label{CH_complete}
 In a metric space $(X, d)$, the space of continuous functions $\mathcal{C}(X)$ equipped with the Hausdorff metric $\mathcal{D}_H$ is a complete metric space.
\end{Lemma}

\begin{Lemma}[Lemma 3.3, \cite{pant2018fixed}]\label{Tn_Suzuki}
Let $(X,d)$ be a metric space and $T_p:\mathcal{C}(X)\rightarrow \mathcal{C}(X)~(p=1,2,\ldots,p)$ continuous STGPC mapping, $\forall~\mathcal{X}, \mathcal{Y}\in \mathcal{C}(X)$,
\begin{align*}
     \frac{1}{2} \mathcal{D}_H(\mathcal{X},T_p(\mathcal{X}))\leq \mathcal{D}_H(\mathcal{X}, \mathcal{Y})~~ \mbox{implies}~~ \mathcal{D}_H(T_p(\mathcal{X}), T_p(\mathcal{Y}))\leq \varphi_n(M_{T_p}(\mathcal{X}, \mathcal{Y})).   
\end{align*}
Define $T:\mathcal{C}(X)\rightarrow \mathcal{C}(X)$ by $T(\mathcal{X})=T_1(\mathcal{X})\bigcup T_2(\mathcal{X}) \bigcup \ldots \bigcup T_p(\mathcal{X})=\displaystyle\bigcup^{p}_{n=1}T_n (\mathcal{X})$ for each $\mathcal{X}\in \mathcal{C}(X)$. Then T also satisfies
\begin{align*}
    \frac{1}{2} \mathcal{D}_H(\mathcal{X},T(\mathcal{X}))\leq \mathcal{D}_H(\mathcal{X}, \mathcal{Y})~~ \mbox{implies}~~ \mathcal{D}_H(T(\mathcal{X}), T(\mathcal{Y}))\leq \eta (M_T(\mathcal{X}, \mathcal{Y})),
\end{align*}
for all $\mathcal{X}, \mathcal{Y}\in \mathcal{C}(X)$, where $\eta=\max \lbrace \varphi_n : n=1,2,\ldots,p \rbrace$.
\end{Lemma}


We can now define the following necessary maps to construct an iterated function system with STGPC. For $p = 1, 2,\ldots, P$, let $L_p : [y_0, y_p] \rightarrow [y_{p-1}, y_p]$ be P contraction homeomorphisms satisfying  \eqref{affine}. Consider the continuous mappings $F_p: K\rightarrow \mathbb{R}$, where $K := A\times \mathbb{R}$. Assume that $F_p$ are continuous in the first variable and STGPC in the second variable, and  $\forall y\in A, z,t \in \mathbb{R}$
\begin{align}\label{fn_contraction}
If~~ \frac{1}{2} d(y,F_p(y,z))\leq d(y,z)~~then~~ d(F_p(y,z),F_p(y,t)) \leq \varphi (m(z,t)),
\end{align}
where $m(z,t)=\max\lbrace d(z,t), d(z,F_p(y,z)), d(t,F_p(y,t)) \rbrace$. Now, the IFS using STGPC is constructed as $\lbrace K; v_p : p = 1, 2,\ldots,P\rbrace$, where the STGPC $w_p : K\rightarrow K$ are defined by
\begin{align}\label{w_p}
    w_p(y,z)=(L_p(y), F_p(y,z)),
\end{align}
such that $w_p(y_0, z_0)=~(y_{p-1},z_{p-1}), w_p(y_P,z_P)=(y_{p},z_{p})$ for $p=1,2,\ldots,P$.
The theorem below presents an attractor for IFS, depicted as the graph of the required FIF through STGPC.
\begin{theorem}
Let $\lbrace(y_p, z_p) : P = 1, 2,\ldots, P\rbrace$ be the given data set and $\lbrace K; w_p : p = 1, 2,\ldots,P\rbrace$ denote its associated IFS with Suzuki-type generalized $\varphi$- contractions $w_p$. Let $G$ denote the attractor of the IFS, then $G$ is the graph $G_g$ of continuous function $g:[y_0, y_P]\rightarrow \mathbb{R}$ satisfying $G(y_p)=z_p$ for all $p=1,2\ldots,P$. That is 
\begin{align*}
G_g=\lbrace (y, g(y)):y\in [y_0, y_P]\rbrace.    
\end{align*}
Subsequently, the IFS possesses an attractor known as $G_g$, such that
\begin{align*}
G_g=\bigcup^{P}_{p=1} w_p(G_g).
\end{align*}
\end{theorem}
\begin{proof}
We show that the graph $G_g$ of $g$ is an attractor of the IFS. Since, $w_p= (L_p(y),F_p(y,z))~\forall~p=1,2,\ldots,P$. From \eqref{f_Fn} and \eqref{w_p}, we get 
 \begin{align*}
     \bigcup^{P}_{p=1} w_p(G_g)&= \bigcup^{P}_{p=1} w_p(\lbrace (y, g(y)):y\in [y_0, y_P] \rbrace),\\
     &= \bigcup^{P}_{p=1} (\lbrace (L_p(y), F_p(y, g(y))):y\in [y_0, y_P] \rbrace),\\
     &= \bigcup^{P}_{p=1} (\lbrace (L_p(y),  g(L_p(y))):y\in [y_0, y_P] \rbrace),\\
     &= \bigcup^{P}_{p=1} (\lbrace (y, g(y)):y\in [y_{p-1}, y_p] \rbrace),\\
     &=G_g.
 \end{align*}    
\end{proof}

\subsection{Construction of $\alpha$-fractal function}\label{S_FIF}
Inspired by Barnsley's \cite{barnsley1986fractal} groundbreaking theory on fractal interpolation functions, Navascués \cite{navascues2005fractal} unveiled an exciting new family of FIFs called $\alpha$-fractal functions, representing a continuous function.  Based on this groundwork, we now develop a more generalized $\alpha$-fractal function using a Suzuki-type generalized $\varphi$-contraction mapping.

\begin{theorem}
Consider $A=[y_0,y_P]$ and $A_p=[y_{p-1},y_p]$. Let the given data set be $\lbrace (y_p,z_P)\in A\times \mathbb{R}, p=1,\ldots,P \rbrace$ with strictly abscissae. Let $L_p:A\rightarrow A_p$ be the $p$ contraction homeomorphisms defined by \eqref{affine} and $F_p(y,z)= \alpha_p z + q_p(y),\, y \in A_p, ~p\in P,$ where $F_p : A\times\mathbb{R}\rightarrow \mathbb{R}$, be the functions that are continuous in the first variable and STGPC map in the second variable satisfying  \eqref{F_p}. Then the operator $\mathbb{T} :\mathcal{C}_0(A ) \rightarrow \mathcal{C}_0(A )$ defined by
\begin{align*}
    (\mathbb{T}h)(y)=(F_p(L_p^{-1}(y)),h(L^{-1}_{p}(y))), ~~y\in A_p, p =1,\dots, P,
\end{align*}
will have a fixed point in $\mathcal C(A)$, where $\mathcal{C}_0(A )$ is the subspace of $\mathcal C(A)$ and $h$ is a continuous function $h: A \rightarrow \mathbb{R}$ with norm $\|h\|=|h|_{\infty} = \max_{y \in A} | h(y)|$ satisfying $h(y_0) = z_0, h(y_P ) = z_P$.
\end{theorem}

\begin{proof}
 Let $\mathcal{C}_0(A )$ be the subspace of $\mathcal C(A )$ with $h(y_0) = z_0, h(y_P ) = z_P$ where  $h$ is a continuous function $h: A \rightarrow \mathbb{R}$ with norm $\|h\|=|h|_{\infty} = \max_{y \in A} | h(y)|$. We note that since $L_p: A \to A_p$ is a homeomorphism, then it is possible to also write  
$|h|_{\infty} = \max_p \max_{y \in A_p} |h(L_p^{-1}(y))|$ since $L^{-1}_p : A_p \to A$ and consequently $\max_{y \in A_p} |h(L_p^{-1}(y))|=  \max_{y \in A} | h(y)|$ for all $p =1,\dots,P$.
 
 Define the RB-operator $\mathbb{T} :\mathcal{C}_0(A ) \rightarrow \mathcal{C}_0(A )$ by
\begin{align}\label{def_T}
    (\mathbb{T}h)(y)=(F_p(L_p^{-1}(y)),h(L^{-1}_{p}(y))), ~~y\in A_p, p =1,\dots, P.
\end{align}
We will proceed as follows to show that $\mathbb{T}$ is an STGPC mapping. We have  
\begin{align*} 
|\mathbb{T}g-\mathbb{T}h|_{\infty}=\max\limits_p\lbrace \max_{y \in A_p}|(F_p(L_p^{-1}(y)),g(L^{-1}_{p}(y)))-(F_p(L_p^{-1}(y)),h(L^{-1}_{p}(y)))| \rbrace.
\end{align*}
Suppose we have $\frac{1}{2} |g - \mathbb{T}g |_{\infty} \leq |g - h|_{\infty}$ then it follows that 
\begin{equation}
\frac{1}{2} | g(L^{-1}_p (y)) - F_p (L_p^{-1} (y), g(L_p^{-1} (y))) | \leq |g - h|_{\infty} \label{neqn:1}
\end{equation}
for $y \in A_p$, $p=1,\dots,P$. For all $p \in \{1,\dots,P\}$ there exists $\bar y_p \in A_p$ such that 
\begin{equation}\label{neqn:2}
|g - h|_{\infty}=   \max_{y \in A_p} |h(L_p^{-1}(y))- g(L_p^{-1}(y))| = |h(L_p^{-1}(\bar y_p))- g(L_p^{-1}(\bar y_p))|.
\end{equation}
Consequently we have 
\[
\frac{1}{2} | g(L^{-1}_p (\bar y_p)) - F_p (L_p^{-1} (\bar y_p), g(L_p^{-1} (\bar y_p))) | \leq  |h(L_p^{-1}(\bar y_p))- g(L_p^{-1}(\bar y_p))|
\]
and hence it follows from $F_p(y,z)$ being Suzuki type generalized $\varphi$ contraction mapping, which satisfies the eq. (\ref{fn_contraction}) that 
\begin{align}
|F_p(L_p^{-1}(\bar y_p),g(L^{-1}_{p}(\bar y_p)))-F_p(L_p^{-1}(\bar y_p),h(L^{-1}_{p}(\bar y_p)))| 
&\leq \varphi(m(g(L^{-1}_{p}(\bar y_p)),h(L^{-1}_{p}(\bar y_p)))) \nonumber\\
&\leq \max_{p} \max_{y \in A_p} \varphi(m(g(L^{-1}_{p}(y)),h(L^{-1}_{p}(y))))  \label{max_f}
\end{align}
where 
\begin{align}\label{phi_m}
 m(g(L^{-1}_{p}(y)),h(L^{-1}_{p}(y)))= (|g(L^{-1}_{p}(y))-&h(L^{-1}_{p}(y))|, |g(L^{-1}_{p}(y))-F_p(L_p^{-1}(y),g(L^{-1}_{p}(y)))|,\\ \nn&  |h(L^{-1}_{p}(y))-F_p(L_p^{-1}(y),h(L^{-1}_{p}(y)))|).
\end{align}
We now use the specific structure of $F_p$ given in (\ref{ln_Fn}) as 
$F_p(y,z)= \alpha_p z + q_p(y)$ to observe that, using \eqref{neqn:2}, we have (for any $p =1,\dots,P$) 
\begin{equation}\label{neqn:4}
|F_p(L_p^{-1}(\bar y_p),g(L^{-1}_{p}(\bar y_p)))-F_p(L_p^{-1}(\bar y_p),h(L^{-1}_{p}(\bar y_p)))| = |\alpha_p| |g(L^{-1}_{p}(\bar y_p)) - h(L^{-1}_{p}(\bar y_p))| =  |\alpha_p|   |g - h|_{\infty}.
\end{equation}
Next note that for any $p' = 1,\dots,P$ we have 
\begin{align}\label{neqn:5}
 \max_{y \in A_{p'}} |F_p(L_{p'}^{-1}(y),g(L^{-1}_{p'}(y)))-F_p(L_{p'}^{-1}(y),h(L^{-1}_{p'}(y)))| 
& =  |\alpha_{p'}|\max_{y \in A_{p'}}  |g(L^{-1}_{p'}(y)) - h(L^{-1}_{p'}(y))| \nonumber \\
& = |\alpha_{p'}| |g - h|_{\infty}.
\end{align}
By using equations \eqref{max_f}, \eqref{phi_m},  \eqref{neqn:5}   and \eqref{neqn:4}, and the monotonicity of $\varphi (\cdot)$, we can conclude that for all  $p'=1,\dots, P$
\begin{align*}
  \max_{y \in A_{p'}}   |F_{p'}(L_{p'}^{-1}(y),g(L^{-1}_{p'}(y))) & -F_{p'}(L_{p'}^{-1}(y),h(L^{-1}_{p'}(y)))|  
 \\
 &\leq  \max\limits_p(  \varphi( \max\limits_{y \in A_p} (|g(L^{-1}_{p}(y))-h(L^{-1}_{p}(y))|, |g(L^{-1}_{p}(y))-F_p(L_p^{-1}(y), g(L^{-1}_{p}(y)))|,\\
 &~~~~~~~~~~~~~~~~~~~~~~~~~~~~~~~|h(L^{-1}_{p}(y))-F_p(L_p^{-1}(y) , h(L^{-1}_{p}(y)))| ))\\
&\leq \varphi (\max (|g-h|_{\infty}, |g-\mathbb{T}g|_{\infty},  |h-\mathbb{T}h|_{\infty} ) )\\
&= \varphi(m(g,h)),\\
\mbox{i.e.},~~~~~~~  |\mathbb{T}g-\mathbb{T}h|_{\infty}&\leq  \varphi (m(g,h)) \qquad \text{on taking the maximum over 
 } p'=1,\dots,P, 
\end{align*}
where  $m(g,h)=\max (|g-h|_{\infty}, |g-\mathbb{T}g|_{\infty},  |h-\mathbb{T}h|_{\infty} )$. Finally we note that this  implies that $\mathbb{T}$ has a unique fixed point, denoted as $g \in \mathcal{C}_0(A)$ such that 
\begin{align}\label{T_func}
    g(y)=\mathbb{T}(g(y))=&F_p(L_p^{-1}(y),g(L^{-1}_{p}(y))), ~~y\in A_p, p=1,2,\ldots,P.\\
    g(L_p(y))=&F_p(y,g(y)), y\in A_p, p=1,2,\ldots,P.\nn
\end{align}

\end{proof}

The $\alpha$-fractal function is constructed using a widely researched fractal interpolation function defined by the equation \eqref{ln_Fn} as detailed in \citep{navascues2005fractal,navascues2005fractal1}.
\begin{align}\label{qn}
q_p(y)=g\circ L^{-1}_{p}(y)-\alpha_p b(y),~~p=1,2,\ldots,P
\end{align}
where $b\in \mathcal{C}(A)$ with $b(y_0)=g(y_0), b(y_p)=g(y_p)$ and $b\neq g$. The function $b$ is known as the base function. Navascués \cite{navascues2005fractal} observed that the base function $b$ linearly depends on $g$. In particular, we consider 
\begin{align}\label{base}
    b=Lg,
\end{align}
where $L:\mathcal{C}(A) \rightarrow \mathcal{C}(A)$ is a bounded linear operator. Now from \eqref{ln_Fn}, \eqref{def_T} and \eqref{qn} we have that,
\begin{align*}
 (\mathbb{T}h)(y)&=\alpha_p\circ h\circ L^{-1}_{p}(y)+g\circ L_p(y)\circ L^{-1}_{p}(y)-\alpha_p b\circ L^{-1}_{p}(y)\\
 &=\alpha_p\circ h\circ L^{-1}_{p}(y)+g-\alpha_p b\circ L^{-1}_{p}(y)\\
 &=g+\alpha_p(h-b)\circ L^{-1}_{p}(y), \forall~ y\in A,~ p=1,2,\ldots,P.
\end{align*}
Hence, $g^{\alpha}$ satisfies the following self-referential equation
\begin{align}\label{SFF}
g^{\alpha}=g+\alpha_p(g^{\alpha}-b)\circ L^{-1}_{p}(y),~~\forall~ y\in A,~ p=1,2,\ldots,P.
\end{align}
The fractal function $g^\alpha$ as defined in \eqref{SFF} can be used to construct the fractal operator, which can approximate any continuous function, defined as follows: 

\begin{definition}[Fractal Operator]
Let  $g^{\alpha}$ be the continuous function defined by the IFS defined by \eqref{ln_Fn}, \eqref{qn} and \eqref{base}. $g^{\alpha}$ is the $\alpha$-fractal function associated with $g$ with respect to $L$. Let $\Omega:y_0<y_1\cdot<y_P$ be the partition of interval $A=[y_0,y_P]$ and define, $|\alpha|_{\infty}= \max \lbrace |\alpha_p|:p=1,2,\ldots,P \rbrace$. The operator with respect to $\Omega$ and $L$ as follows:
\end{definition} 
\begin{align*}
 \mathcal{G}^{\alpha}: &~ \mathcal{C}(A)\rightarrow \mathcal{C}(A)\\
 &g\longmapsto g^{\alpha}.
\end{align*}
The map $\mathcal{G}^{\alpha}$ is a bounded linear operator referred to as $\alpha$-fractal operator with respect to partition $\Omega$ and base function $b=Lg$. The uniform error bound for the perturbation process is given in \cite{navascues2010fractal} as
\begin{align*}
||g^{\alpha}-g||_{\infty} \leq \frac{||\alpha||_{\infty}}{1-||\alpha||_{\infty}} ||g-b||_{\infty}.
\end{align*}

We have thus, constructed the FIF through STGPC described in \eqref{SFF}, and we will visualize the graph of FIF for the given data set. The following subsection presents the graphical representation of $\alpha$-FIF through STGPC. 

\subsection*{Graph of $\alpha$-FIF through STGPC}
 In this case, we are using the $g(y)$ function constructed in the example \ref{ex_conti} and the base function $b=g(y^2)$. Consider the initial dataset is as follows: $\lbrace(4,0),(5,2),(7,1),(7.5,0.5),(8,0),(9,0), (10, 0) \rbrace$. 
The graph of FIF using STGPC is shown in Figure \ref{S_graph}. The MATLAB version R2023b software package is used to construct the FIF with the given dataset.
\begin{figure}[htp]
    \centering
    \includegraphics[width=0.5\linewidth]{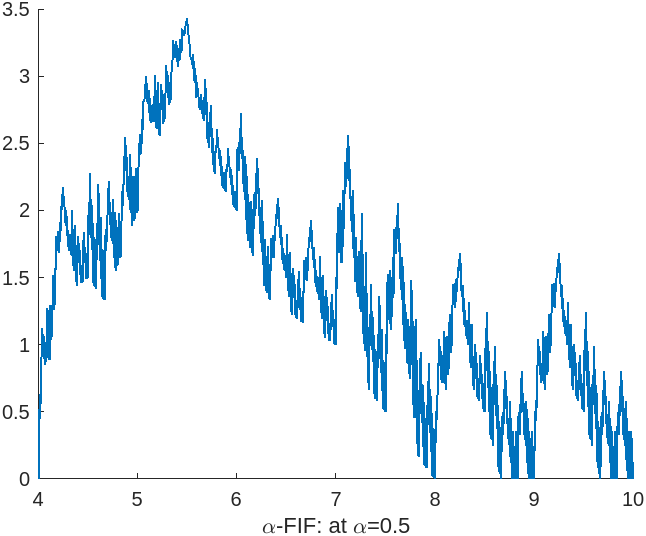}
    \caption{Graph of $\alpha$-FIF through STGPC for $\alpha=0.5$ for the given data set}
    \label{S_graph}
\end{figure}

In Figure \ref{S_graph}, fluctuations are noticeable for the given dataset. These fluctuations can vary depending on the values of the scaling factor $\alpha$. To analyze the fractal patterns and trend fluctuations, we will use box-dimensional analysis. The increase in box dimension indicates highly fluctuating trends, which we will thoroughly explain in the following subsection. 

\subsection{Dimensional Analysis}\label{B_D}
The fractal dimension is important for understanding the complexity of fractals. The box-counting dimension is used to characterize the scaling properties of fractals by demonstrating how the detail or complexity of the fractal changes as we zoom in.  The box-counting dimension is defined as follows:
\begin{definition}[Box Dimension \citep{navascues2005fractal}] 
Let $Z$ be a non-empty bounded subset of the metric space $(X,d)$. The box dimension of $Z$ is defined as 
\begin{align*}
\dim_B Z=\lim_{\epsilon\to 0} \frac{\log N_{\epsilon}(Z)}{-\log \epsilon},
\end{align*}
The expression $\log N_{\epsilon}(Z)$ represents the smallest number of sets with a diameter of at most $\epsilon$ required to cover $Z$, assuming the limit exists. If the limit does not exist, the upper and lower box dimensions are defined as follows:
\begin{align*}
    \overline{\dim}_B Z&= \limsup_{\epsilon\to 0} \frac{\log N_{\epsilon}(Z)}{-\log \epsilon},\\
        \underline{\dim}_B Z&= \liminf_{\epsilon\to 0} \frac{\log N_{\epsilon}(Z)}{-\log \epsilon}.
\end{align*}
\end{definition}
For a comprehensive analysis of box dimension, the reader may refer \cite{falconer2014fractal}.

The following theorem presents a significant method for calculating the box dimension of $\alpha$-FIF for the given dataset. 

\begin{theorem}\label{dim}[Theorem 5, \citep{barnsley1989hidden}] Let $\Omega=(y_0,y_1,\ldots,<y_P)$ be a partition of $A=[y_0,y_P]$ satisfying $y_0<y_1<,\ldots,<y_P$ and let $\alpha=(\alpha_1,\alpha_2,\ldots,\alpha_P)$. Let $g$ and $b$ be Lipschitz functions defined on $A$ with $b(y_0)=g(y_0)$ and $b(y_P)=g(y_P)$. If the data points $\lbrace (y_i,g(y_i)):i=0,1,\ldots,P \rbrace$ are not collinear, then 
\begin{align*}
\dim_B (G_r(g^{\alpha}_{\Omega,b}))=
\begin{cases}
     &D,~ if \sum^{P}_{i=1} |\alpha_i|>1; \\
     &1,~ otherwise,
\end{cases}
\end{align*}
where $(G_r(g^{\alpha}_{\Omega,b}))$ denotes the graph of $g^{\alpha}_{\Omega,b}$. The box dimension, denoted as $D$, can be determined by solving the following equation. 
\begin{align*}
    \sum^{P}_{i=1} |\alpha_i| a_i^{D-1}=1,
\end{align*}
where $a_i=y_i-y_{i-1}$, for all $i=0,1,\ldots,P$.
\end{theorem}

The box dimension analysis will be used to examine the case study on Spinach price volatility within a specific time period. The dimensional analysis is directly linked to price variations over a particular time frame.

In the upcoming section, we will utilize the fractal interpolation function constructed in section \ref{S_FIF} to examine the price fluctuations of spinach within the vegetable market. This function adheres to the Suzuki conditions \eqref{Suzuki_def} for a defined $\varphi$. As a type of non-linear contraction, it allows us to find a fixed point without depending on continuity or semi-continuity, thereby extending the principles of Banach's fixed point theory. Given that we are employing interpolation theory, we treat the fractal function as a continuous function to interpolate the data set. Finally, we construct the fractal interpolation function for the different values of scaling factor $\alpha$'s and discuss the box dimension analysis to quantify the complexities of price variability.

\section{Case study}\label{cs}

The fractal interpolation function (FIF) is widely utilized across various disciplines, including finance, time series data analysis, and examining COVID-19 spread through box dimension analysis. Verma and Kumar \cite{kumar2024alpha}  analyzed the stock index S\&P utilizing FIF. Additionally, \cite{verma2024fractal} provided a dimensional analysis of FIF to examine the stock prices of financial institutions post-mergers and acquisitions. The dimensional study of COVID-19 spread was detailed in the work of \cite{agrawal2023dimensional}. The applications of FIF span diverse areas of study.
In the following case study, we analyze price fluctuations of Spinach over a fixed period.


The vegetable market is known for its extremely unpredictable nature, which can be attributed to a range of factors such as weather conditions, disruptions in the supply chain, seasonal shifts, and fluctuations in market demand. In this fractal approach, we use historical average price data to analyze the price variation over a period through FIFs. The FIF method efficiently grasps the intricate details of price changes, making it an essential tool for predicting future market conditions. Furthermore, incorporating a scaling factor enhances flexibility and provides valuable insights for analyzing vegetable prices using FIFs.  

In the following sections, the methodology for the case study is presented in \ref{MTH}. Spinach's average monthly price data is discussed in section \ref{D_A}. Section \ref{F_A} presents the fractal analysis of Spinach's price data. Finally, section \ref{BDA} discusses the dimensional analysis of FIFs to understand the complexity of the historical data set.

\subsection{Methodology}\label{MTH}
The analysis of spinach price fluctuation is outlined in the following sections. The flowchart in Fig. \ref{flow}  illustrates the methodology for studying price fluctuation using FIFs.
\begin{figure}[htp]
    \centering 
    \includegraphics[width=0.4\linewidth]{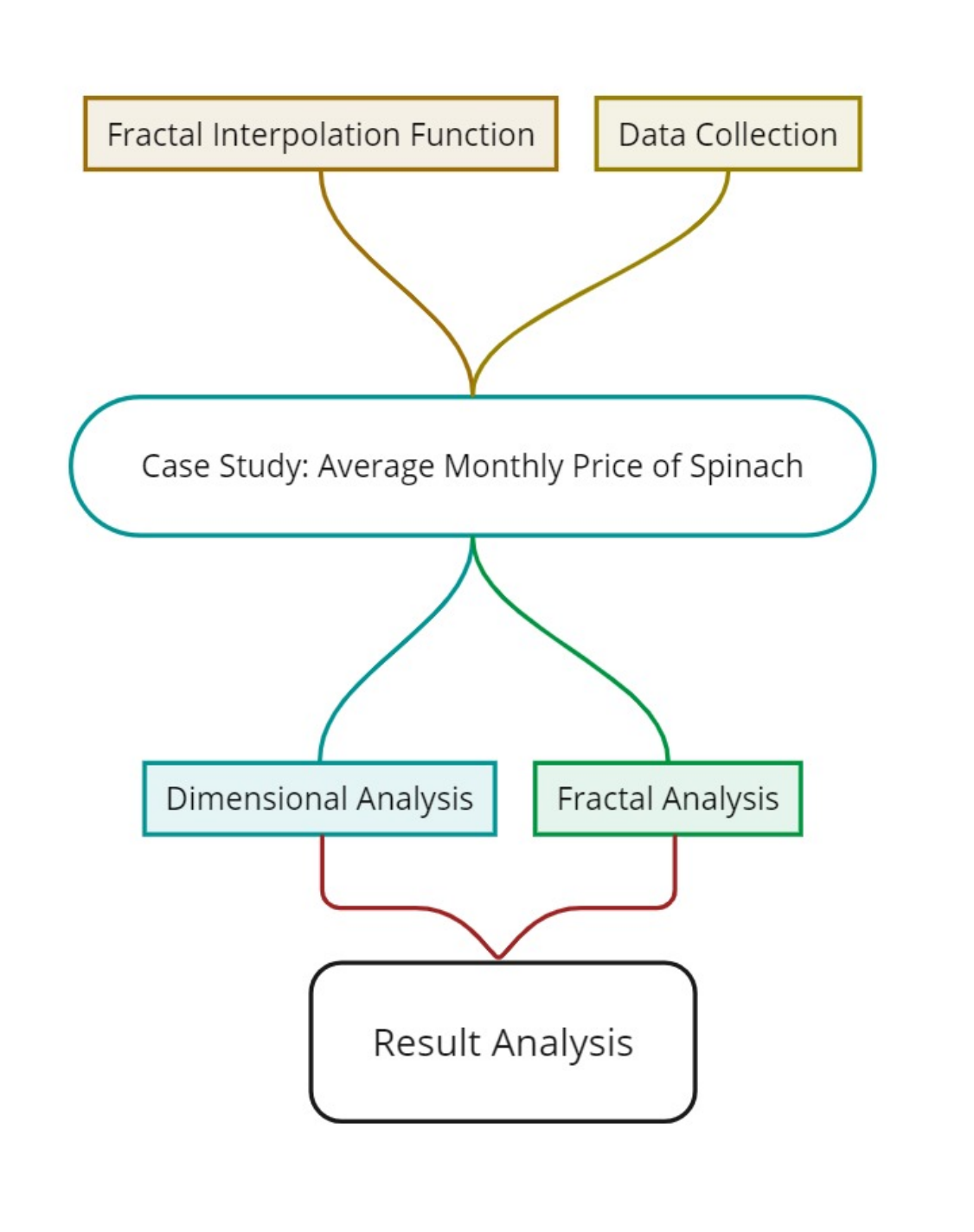}
    \caption{Methodlogy flowchart }
    \label{flow}
\end{figure}

\subsection{Case study data}\label{D_A}
For this specific scenario, we focus on Spinach and examine the average monthly price data from September 2023 to July 2024 in Azadpur market, New Delhi, India. All prices are listed in \rupee (Indian rupees) per kg. All the data in Table \ref{price_table} is sourced from the Government of India's agricultural market website at \url{https://agmarknet.gov.in/PriceTrends/SA_Month_PriM.aspx}.

\begin{table}[htp]
    \centering
    \begin{adjustbox}{width=\textwidth} 
    \begin{tabularx}{0.8\textwidth}{l*{5}{X}}\toprule
    Month & $y_i$    & Min Price & Max Price & Average Price $(z_i)$ \\     
      \midrule
September 2023& 0.0  & \rupee 5 & \rupee 11  &\rupee 8.0 \\
       October 2023 & 0.1 & \rupee 5  & \rupee 10 &\rupee 7.5 \\
       November 2023 & 0.2 & \rupee 3 & \rupee 9  &\rupee 6.0\\
       December 2023 & 0.3 & \rupee 4 & \rupee 10 &\rupee 7.0\\
       January 2024 & 0.4 & \rupee 5 &\rupee 15 &\rupee 10\\
       February 2024 & 0.5 & \rupee 2  &\rupee 8  &\rupee 5.0\\
       March 2024 & 0.6  & \rupee 4  & \rupee 10  &\rupee 7.0 \\
       April 2024 & 0.7 & \rupee 3  &\rupee 8  &\rupee 5.5 \\
       May 2024 & 0.8  & \rupee 5 & \rupee 10 & \rupee 7.5\\
       June 2024 & 0.9 & \rupee 7 & \rupee 10 & \rupee 8.5\\
       July 2024 & 1.0  &\rupee 5 & \rupee 15  & \rupee 10 \\
        \bottomrule
    \end{tabularx}
    \end{adjustbox}
    \caption{Average price of spinach (in \rupee) per kilogram for each month}
    \label{price_table}
\end{table}

\subsection{Fractal Analysis}\label{F_A}
 We will study the price variation of Spinach over a given period of time from the fractal perspective. This will involve the $\alpha-$FIF, which comprises the initial data set, the function $g(y)$, the base function $b$, and a scaling factor $\alpha$ as defined in \eqref{S_FIF}. In this case, the base function $b$ is taken as $b=g(y^2)$, the function $g(y)$ is taken as \eqref{g_f} with $b\neq g$, and the initial data set is given by $\lbrace (y_i,z_i):i=0,1,\ldots,10 \rbrace$ which is not collinear. Finally, the analysis of box dimension as described in Section \ref{FIF_S} is presented to study the complexity of price fluctuations. 
 
\begin{equation}\label{g_f}
g(y)=
    \begin{cases}
        -5y + 8 & \text{for } 0.0 \leq y < 0.1 \\
-15y + 9 & \text{for } 0.1 \leq y < 0.2 \\
10y + 4 & \text{for } 0.2 \leq y < 0.3 \\
30y - 2 & \text{for } 0.3 \leq y < 0.4 \\
-50y + 30 & \text{for } 0.4 \leq y < 0.5 \\
20y - 5 & \text{for } 0.5 \leq y < 0.6 \\
-15y + 16 & \text{for } 0.6 \leq y < 0.7 \\
20y - 8.5 & \text{for } 0.7 \leq y < 0.8 \\
10y+0.5 & \text{for } 0.8 \leq y < 0.9 \\
15y - 5 & \text{for } 0.9 \leq y \leq 1.0.
    \end{cases}
\end{equation}

By varying $\alpha$, different fractal interpolation functions are obtained. For $\alpha=0.4$, the FIF is represented in the following figures, which show the trend of average price for the different values of $\alpha$. Figures \ref{P3}, \ref{P5}, and \ref{Pmix} depict the FIF, while \ref{P0} depicts the particular case of FIF, which is the classical interpolation function. The variation in the vertical scaling factor $\alpha$ leads to changes in the peaks and troughs of the FIFs graphs. Increasing the values of $\alpha$ will result in higher peaks and troughs, making the graphs more complex. To measure this complexity, we calculate the box dimension for different values of $\alpha$.






\begin{figure}[htbp]
    \centering
    \subfloat[$\alpha$-FIF  $g^\alpha$ at $\alpha=0.4$ for the data set $(y_i,z_i)$]{
        \includegraphics[width=0.4\textwidth]{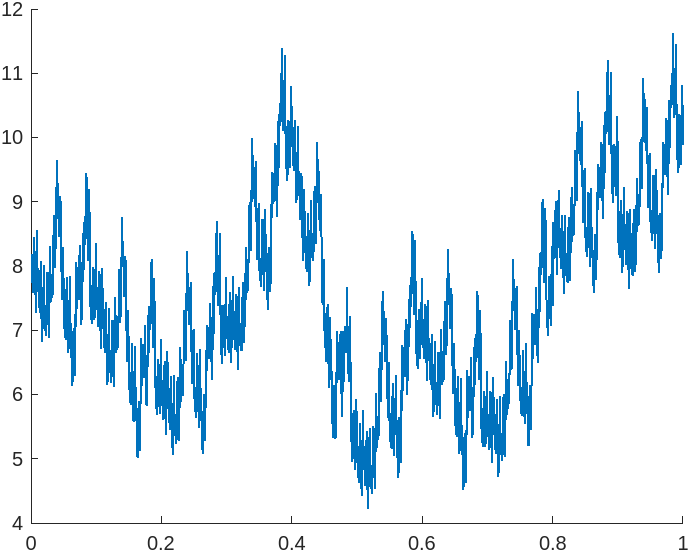}
        \label{P3}
    }
    \hspace{0.05\textwidth} 
    \subfloat[$\alpha$-FIF  $g^\alpha$ at $\alpha=0.6$ for the data set $(y_i,z_i)$]{
        \includegraphics[width=0.4\textwidth]{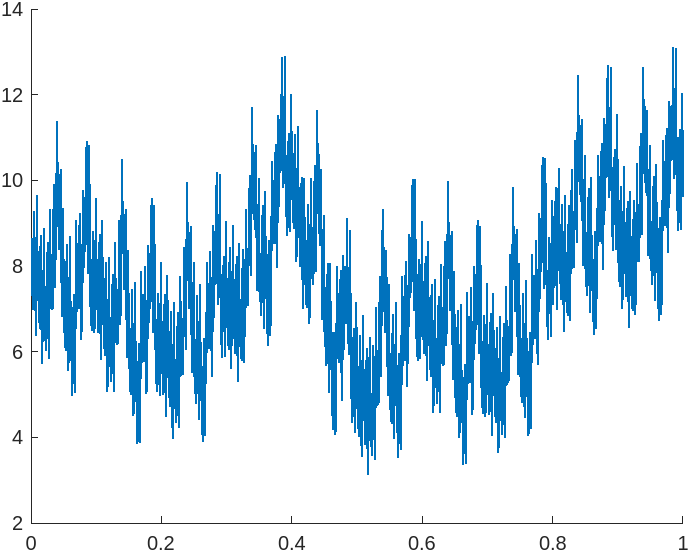}
        \label{P5}
    }
    \hspace{0.05\textwidth}
    \subfloat[$\alpha$-FIF  $g^\alpha$ at $\alpha=(0.1, 0.2, 0.5, 0.2, 0.4, 0.2, 0.4, 0.2, 0.3, 0.1)$ for the data set $(y_i,z_i)$]{
        \includegraphics[width=0.4\textwidth]{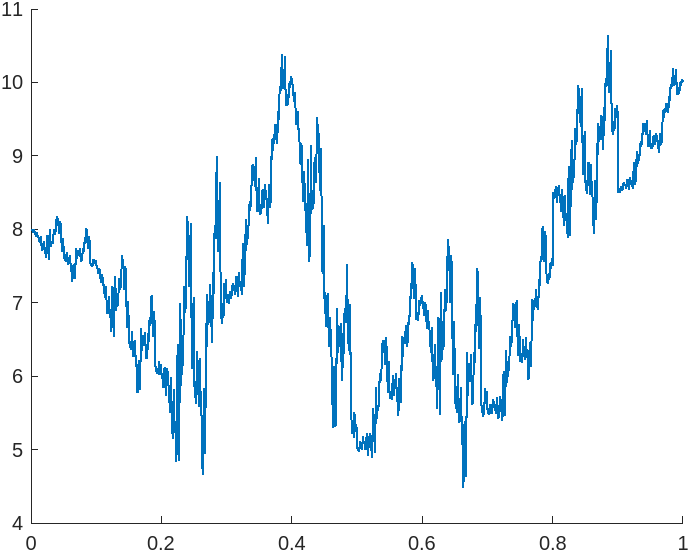}
        \label{Pmix}
    }
        \hspace{0.05\textwidth}
    \subfloat[$\alpha$-FIF  $g^\alpha$ at $\alpha=0.0$ for the data set $(y_i,z_i)$ which represents the classical interpolation function]{
        \includegraphics[width=0.4\textwidth]{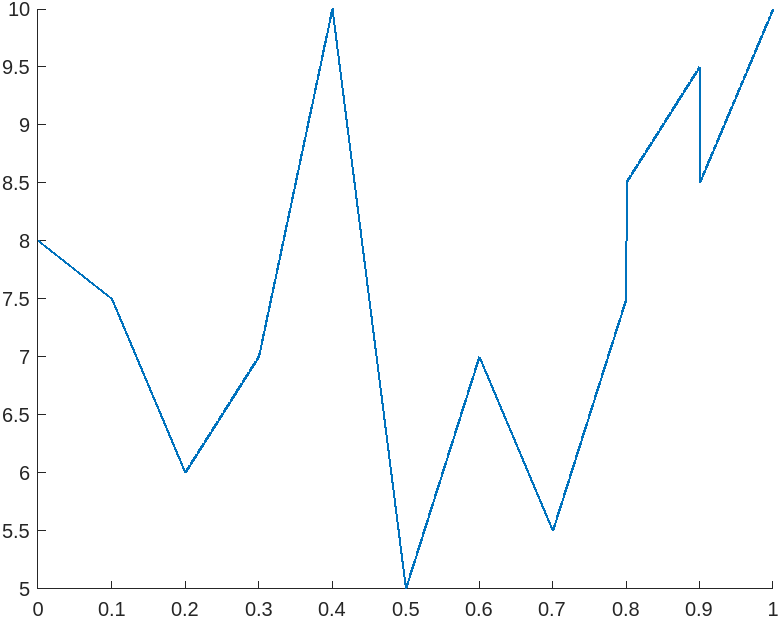}
        \label{P0}
    }
    \caption{$\alpha$-FIF for the data set $(y_i,z_i)$ for the different values of $\alpha$}
    \label{VEG}
\end{figure}

\subsection{Box Dimension Analysis}\label{BDA}

The fractal dimension of the $\alpha$-FIFs with the given datasets can be calculated using Theorem \ref{dim}. Box dimension analysis allows us to study the complexity or fluctuations in the $\alpha$-FIF trends. An increase in box dimension indicates that the trend exhibits more fluctuations over the specified time period. We calculate the box dimension for different values of $\alpha$ given below.

\textbf{Case 1 :} Based on the given values $a_i=0.1$ and $\alpha=0.4$, the fractal function depicted in Figure \ref{P3} is plotted. It's worth noting that $\sum^{10}_{i=0} 0.4=4>1$. Therefore, the dimension will be calculated using Theorem \ref{dim}:
\begin{align*}
\sum^{11}_{i=1} (0.4) \left( \frac{1}{10} \right)^{D-1}=1 \implies 4 \left( \frac{1}{10} \right)^{D-1}=1 \implies 10^{D-1}=4.
\end{align*}
After simplifying the equation, we will have
\begin{align*}
    D=\log(4)+1\approx 1.60.
\end{align*}

\textbf{Case 2 :} Based on the given values $a_i=0.1$ and $\alpha=0.6$, the fractal function depicted in Figure \ref{P5} has been plotted. It's worth noting that $\sum^{10}_{i=0} 0.6=6>1$. Therefore, the dimension will be calculated using Theorem \ref{dim}:
\begin{align*}
\sum^{10}_{i=0} (0.6) \left( \frac{1}{10} \right)^{D-1}=1 \implies 6 \left( \frac{1}{10} \right)^{D-1}=1 \implies 10^{D-1}=6.
\end{align*}
After simplifying the equation, we will have
\begin{align*}
    D=\log(6)+1\approx 1.77.
\end{align*}

\textbf{Case 3 :} Based on the given values $a_i=0.1$ and $\alpha=(0.1, 0.2, 0.5, 0.2, 0.4, 0.2, 0.4, 0.2, 0.3, 0.1)$, the fractal function depicted in Figure \ref{Pmix} has been plotted. It's worth noting that $\sum^{10}_{i=0} (0.1+ 0.2+0.5+0.2+0.4+0.2+0.4+0.2+0.3+ 0.1)=2.6>1$. Therefore, the dimension will be calculated using Theorem \ref{dim}:
\begin{align*}
&\sum^{10}_{i=0} (0.1+ 0.2+0.5+0.2+0.4+0.2+0.4+0.2+0.3+ 0.1) \left( \frac{1}{10} \right)^{D-1}=1\\
&= 2.6 \left( \frac{1}{10} \right)^{D-1}=1 \implies 10^{D-1}=2.6.
\end{align*}
After simplifying the equation, we will have
\begin{align*}
    D=\log(2.6)+1\approx 1.41.
\end{align*}

The complexity of price data is directly affected by the dimension of the graph of Spinach retail prices. This study also includes a comparative analysis of classical and fractal interpolation models. Figures \ref{P3}, \ref{P5} and \ref{Pmix}  illustrate the $\alpha$-FIF for the different values of $\alpha$. The classical interpolation function is obtained by setting $\alpha=0.0$, as shown in the graph \ref{P0}. The classical interpolation function is a specific case of the $\alpha$-FIF.

\section{Conclusion and future work}
The proposed research outlines the construction of the FIF, also known as the $\alpha$-fractal function, using an STGPC with a constant scaling factor. It also covers the establishment of an IFS and discusses the existence and uniqueness of the attractor. The article concludes with a case study illustrating the application of the FIF in analyzing the fractal dimension of spinach prices over time, demonstrating the price variation with different values of $\alpha$ and the complexity of the resulting graph as discussed by the fractal dimensional analysis. This dimensional analysis provides insights into understanding the complexities of spinach prices.

The scope of this work could be broadened to include an examination of the financial context, as it offers a more appropriate depiction of economic data. This approach is particularly well suited for applying fractal methods to investigate complex patterns in financial analysis. Moreover, this method can forecast future trends by analyzing historical datasets.

\section*{Statements and Declaration} 
\noindent\textbf{Conflict of interest:} The authors declare that they have no conflict of interest.\\
\noindent\textbf{Ethical approval:} This article does not contain any studies with human participants or animals performed by any of the authors.\\
\noindent\textbf{Funding:} The authors did not receive support from any organization for the submitted work.\\
\noindent\textbf{Data availability:} \url{https://agmarknet.gov.in/PriceTrends/SA_Month_PriM.aspx}

 \bibliography{sn-bibliography}
\bibliographystyle{acm}
 
\end{document}